\documentclass[12pt]{amsart}
\usepackage{format}

\makeatletter
\newcommand*\bigcdot{\mathpalette\bigcdot@{.7}}
\newcommand*\bigcdot@[2]{\mathbin{\vcenter{\hbox{\scalebox{#2}{$\m@th#1\bullet$}}}}}
\makeatother

\newcommand{\adjo}{\mathrel{\vcenter{\offinterlineskip \ialign{##\cr$\rightarrow$\cr\noalign{\kern 0pt}$\leftarrow$\cr}}}}

\newcommand{\mf}[1]{{\mathfrak{#1}}}
\newcommand{\mb}[1]{{\mathbb{#1}}}
\newcommand{\mc}[1]{{\mathcal{#1}}}
\newcommand{\mrm}[1]{{\mathrm{#1}}}

\DeclareMathOperator{\mhm}{MHM}

\DeclareMathOperator{\Gr}{Gr}

\definecolor{unitary}{rgb}{0, 0.6, 0.2}

\title{The FPP Conjecture for Real Reductive Groups}

\author{Dougal Davis}

\author{Lucas Mason-Brown}

\begin{document}


\begin{abstract}
In this paper, we prove the FPP conjecture, giving a strong upper bound on the unitary dual of a real reductive group. Our proof is an application of the global generation properties of $\mathcal{D}$-modules on the flag variety and their Hodge filtrations.
\end{abstract}

\maketitle

\tableofcontents

\section{Introduction}

Let $G(k)$ be the $k$-points of a connected reductive algebraic group $G$ defined over a local field $k$. Let $\Pi_a(G(k))$ (resp. $\Pi_u(G(k))$, $\Pi_t(G(k))$) denote the set of equivalence classes of irreducible admissible (resp. unitary, tempered) complex representations of $G(k)$. Then there are inclusions
$$\Pi_t(G(k)) \subset \Pi_u(G(k)) \subset \Pi_a(G(k))$$
The Langlands classification (\cite{Langlands1989},\cite{Silberger1978}) reduces the classification of $\Pi_a(G(k))$ to the classification of $\Pi_t(M(k))$ for $k$-Levi subgroups $M$ of $G$. More precisely, fix a minimal $k$-parabolic subgroup $P_0 =M_0N_0 \subset G$. For any $k$-Levi subgroup $M$ of $G$, let $\mb{X}^*(M)_k$ denote the abelian group of algebraic characters $M \to \mathbb{G}_m$ defined over $k$, and let $\fa_M^* = \mb{X}^*(M)_k \otimes_{\ZZ} \RR$. Then $\fa_M^*$ can be identified with the set of continuous characters $M(k) \to \RR_{>0}$. The Langlands classification is a bijection
$$\Pi_a(G(k)) \leftrightarrow \{(P, \sigma, \nu)\}$$
where $P=MN$ is a $k$-parabolic subgroup containing $P_0=M_0N_0$, $\sigma$ is an irreducible tempered representation of $M(k)$ (considered up to equivalence), and $\nu$ is an element of $\fa_M^*$ satisfying a certain positivity condition with respect to $N$. This bijection takes a triple $(P,\sigma,\nu)$ to the unique irreducible quotient $J(P,\sigma,\nu)$ of the induced representation $\Ind^{G(k)}_{P(k)} (\sigma \otimes \nu)$.  

There is a general philosophy that unitary representations should be `small' deformations of tempered representations. In terms of the Langlands classification, this means:

\begin{philosophy*}
Fix a pair $(P,\sigma)$ as above and let
$$U(P,\sigma) := \{\nu \in \fa_M^* \mid J(P,\sigma,\nu) \text{ is unitary}\}.$$
Then $U(P,\sigma)$ is a `small' subset of $\fa_M^*$ (in a sense which depends on $P$ and $\sigma$). 
\end{philosophy*}

There are several precise theorems and conjectures of this form. The \emph{Dirac inequality} states that $U(P,\sigma)$ is contained in a small closed ball (with radius depending on $(P,\sigma)$). For real reductive groups, the Dirac inequality was proved by Parthasarathy in \cite{Parthasarathy} using harmonic analysis on symmetric spaces. In \cite{HuangPandzic}, Huang and Pandzic proved a slightly sharper bound of a similar nature using purely algebraic methods. An analogous result was obtained for $p$-adic groups by Barbasch, Ciubotaru, and Trapa (\cite{BCTDirac}). 

Since the unitarity of $J(P,\sigma,\nu)$ can only change at reducibility hyperplanes, it is clear that the region $U(P,\sigma)$ is a union of simplices defined by rational inequalities. From this point of view, the spherical bounds which appear in the Dirac inequality and its cousins seem unnatural.

\begin{wrapfigure}{R}{.5\textwidth}
\begin{center}
\begin{tikzpicture}[scale=0.5]
\draw[fill=Cyan, draw=Cyan, opacity=0.1] (0,0) circle (6.110100926607787);

\fill[fill=red, opacity=0.3] (0,0) -- (-0.0,2.3094010767585034) -- (2.0,5.773502691896258) -- (2.0,3.464101615137755) -- cycle;
\fill[fill=red, opacity=0.3] (0,0) -- (0.0,-2.3094010767585034) -- (2.0,-5.773502691896258) -- (2.0,-3.464101615137755) -- cycle;
\fill[fill=red, opacity=0.3] (0,0) -- (-1.9999999999999998,1.1547005383792515) -- (-3.9999999999999996,4.618802153517006) -- (-1.9999999999999998,3.4641016151377544) -- cycle;
\fill[fill=red, opacity=0.3] (0,0) -- (2.0,-1.1547005383792517) -- (6.000000000000001,-1.1547005383792517) -- (4.0,0.0) -- cycle;
\fill[fill=red, opacity=0.3] (0,0) -- (-2.0,-1.1547005383792517) -- (-6.000000000000001,-1.1547005383792517) -- (-4.0,-0.0) -- cycle;
\fill[fill=red, opacity=0.3] (0,0) -- (1.9999999999999998,1.1547005383792515) -- (3.9999999999999996,4.618802153517006) -- (1.9999999999999998,3.4641016151377544) -- cycle;
\fill[fill=red, opacity=0.3] (0,0) -- (-0.0,-2.3094010767585034) -- (-2.0,-5.773502691896258) -- (-2.0,-3.464101615137755) -- cycle;
\fill[fill=red, opacity=0.3] (0,0) -- (0.0,2.3094010767585034) -- (-2.0,5.773502691896258) -- (-2.0,3.464101615137755) -- cycle;
\fill[fill=red, opacity=0.3] (0,0) -- (1.9999999999999998,-1.1547005383792515) -- (3.9999999999999996,-4.618802153517006) -- (1.9999999999999998,-3.4641016151377544) -- cycle;
\fill[fill=red, opacity=0.3] (0,0) -- (-2.0,1.1547005383792517) -- (-6.000000000000001,1.1547005383792517) -- (-4.0,0.0) -- cycle;
\fill[fill=red, opacity=0.3] (0,0) -- (2.0,1.1547005383792517) -- (6.000000000000001,1.1547005383792517) -- (4.0,-0.0) -- cycle;
\fill[fill=red, opacity=0.3] (0,0) -- (-1.9999999999999998,-1.1547005383792515) -- (-3.9999999999999996,-4.618802153517006) -- (-1.9999999999999998,-3.4641016151377544) -- cycle;

\draw[fill=unitary, draw=unitary] (1.3333333333333333,-0.0) -- (0.6666666666666666,1.1547005383792515) -- (-0.6666666666666666,1.1547005383792515) -- (-1.3333333333333333,-0.0) -- (-0.6666666666666666,-1.1547005383792515) -- (0.6666666666666666,-1.1547005383792515) -- cycle;
\draw[fill=unitary, draw=unitary] (2.0,-1.1547005383792517) -- (1.3333333333333333,-0.0) -- (2.0,1.1547005383792517) -- cycle;
\draw[fill=unitary, draw=unitary] (2.0,1.1547005383792517) -- (0.6666666666666666,1.1547005383792515) -- (0.0,2.309401076758503) -- cycle;
\draw[fill=unitary, draw=unitary] (0.0,2.309401076758503) -- (-0.6666666666666666,1.1547005383792515) -- (-2.0,1.1547005383792517) -- cycle;
\draw[fill=unitary, draw=unitary] (-2.0,1.1547005383792517) -- (-1.3333333333333333,-0.0) -- (-2.0,-1.1547005383792517) -- cycle;
\draw[fill=unitary, draw=unitary] (-2.0,-1.1547005383792517) -- (-0.6666666666666666,-1.1547005383792515) -- (0.0,-2.309401076758503) -- cycle;
\draw[fill=unitary, draw=unitary] (0.0,-2.309401076758503) -- (0.6666666666666666,-1.1547005383792515) -- (2.0,-1.1547005383792517) -- cycle;
\draw[fill=unitary, draw=unitary] (2.0,5.773502691896258) circle (0.1);
\draw[fill=unitary, draw=unitary] (2.0,-5.773502691896258) circle (0.1);
\draw[fill=unitary, draw=unitary] (-3.9999999999999996,4.618802153517006) circle (0.1);
\draw[fill=unitary, draw=unitary] (6.000000000000001,-1.1547005383792517) circle (0.1);
\draw[fill=unitary, draw=unitary] (-6.000000000000001,-1.1547005383792517) circle (0.1);
\draw[fill=unitary, draw=unitary] (3.9999999999999996,4.618802153517006) circle (0.1);
\draw[fill=unitary, draw=unitary] (-2.0,-5.773502691896258) circle (0.1);
\draw[fill=unitary, draw=unitary] (-2.0,5.773502691896258) circle (0.1);
\draw[fill=unitary, draw=unitary] (3.9999999999999996,-4.618802153517006) circle (0.1);
\draw[fill=unitary, draw=unitary] (-6.000000000000001,1.1547005383792517) circle (0.1);
\draw[fill=unitary, draw=unitary] (6.000000000000001,1.1547005383792517) circle (0.1);
\draw[fill=unitary, draw=unitary] (-3.9999999999999996,-4.618802153517006) circle (0.1);

\draw[dashed, draw opacity=0.2] (-6.0,-7) -- (-6.0,7);
\draw[dashed, draw opacity=0.2] (-4.0,-7) -- (-4.0,7);
\draw[dashed, draw opacity=0.2] (-2.0,7) -- (-2.0,-7);
\draw[dashed, draw opacity=0.2] (0.0,7) -- (0.0,-7);
\draw[dashed, draw opacity=0.2] (2.0,-7) -- (2.0,7);
\draw[dashed, draw opacity=0.2] (4.0,-7) -- (4.0,7);
\draw[dashed, draw opacity=0.2] (6.0,7) -- (6.0,-7);
\draw[dashed, draw opacity=0.2] (-3.875644347017859,-7) -- (-7,-5.196152422706632);
\draw[dashed, draw opacity=0.2] (-7,-2.886751345948129) -- (0.1243556529821408,-7);
\draw[dashed, draw opacity=0.2] (4.124355652982141,-7) -- (-7,-0.5773502691896258);
\draw[dashed, draw opacity=0.2] (7,-6.3508529610858835) -- (-7,1.7320508075688774);
\draw[dashed, draw opacity=0.2] (7,-4.041451884327381) -- (-7,4.041451884327381);
\draw[dashed, draw opacity=0.2] (7,-1.7320508075688774) -- (-7,6.3508529610858835);
\draw[dashed, draw opacity=0.2] (7,0.5773502691896258) -- (-4.124355652982141,7);
\draw[dashed, draw opacity=0.2] (7,2.886751345948129) -- (-0.1243556529821408,7);
\draw[dashed, draw opacity=0.2] (3.875644347017859,7) -- (7,5.196152422706632);
\draw[dashed, draw opacity=0.2] (3.875644347017859,-7) -- (7,-5.196152422706632);
\draw[dashed, draw opacity=0.2] (-0.1243556529821408,-7) -- (7,-2.886751345948129);
\draw[dashed, draw opacity=0.2] (7,-0.5773502691896258) -- (-4.124355652982141,-7);
\draw[dashed, draw opacity=0.2] (7,1.7320508075688774) -- (-7,-6.3508529610858835);
\draw[dashed, draw opacity=0.2] (7,4.041451884327381) -- (-7,-4.041451884327381);
\draw[dashed, draw opacity=0.2] (7,6.3508529610858835) -- (-7,-1.7320508075688774);
\draw[dashed, draw opacity=0.2] (4.124355652982141,7) -- (-7,0.5773502691896258);
\draw[dashed, draw opacity=0.2] (0.1243556529821408,7) -- (-7,2.886751345948129);
\draw[dashed, draw opacity=0.2] (-7,5.196152422706632) -- (-3.875644347017859,7);
\draw[dashed, draw opacity=0.2] (-6.625214782339287,-7) -- (-7,-6.3508529610858835);
\draw[dashed, draw opacity=0.2] (-5.291881449005953,-7) -- (-7,-4.041451884327381);
\draw[dashed, draw opacity=0.2] (-7,-1.7320508075688774) -- (-3.95854811567262,-7);
\draw[dashed, draw opacity=0.2] (-7,0.5773502691896258) -- (-2.6252147823392864,-7);
\draw[dashed, draw opacity=0.2] (-7,2.886751345948129) -- (-1.2918814490059531,-7);
\draw[dashed, draw opacity=0.2] (0.04145188432738026,-7) -- (-7,5.196152422706632);
\draw[dashed, draw opacity=0.2] (1.3747852176607136,-7) -- (-6.708118550994047,7);
\draw[dashed, draw opacity=0.2] (-5.374785217660713,7) -- (2.708118550994047,-7);
\draw[dashed, draw opacity=0.2] (4.041451884327381,-7) -- (-4.041451884327381,7);
\draw[dashed, draw opacity=0.2] (-2.708118550994047,7) -- (5.374785217660713,-7);
\draw[dashed, draw opacity=0.2] (-1.3747852176607136,7) -- (6.708118550994047,-7);
\draw[dashed, draw opacity=0.2] (7,-5.196152422706632) -- (-0.04145188432738026,7);
\draw[dashed, draw opacity=0.2] (7,-2.886751345948129) -- (1.2918814490059531,7);
\draw[dashed, draw opacity=0.2] (2.6252147823392864,7) -- (7,-0.5773502691896258);
\draw[dashed, draw opacity=0.2] (7,1.7320508075688774) -- (3.95854811567262,7);
\draw[dashed, draw opacity=0.2] (7,4.041451884327381) -- (5.291881449005953,7);
\draw[dashed, draw opacity=0.2] (7,6.3508529610858835) -- (6.625214782339287,7);
\draw[dashed, draw opacity=0.2] (-7,-6.92820323027551) -- (7,-6.92820323027551);
\draw[dashed, draw opacity=0.2] (-7,-5.773502691896258) -- (7,-5.773502691896258);
\draw[dashed, draw opacity=0.2] (-7,-4.618802153517007) -- (7,-4.618802153517007);
\draw[dashed, draw opacity=0.2] (-7,-3.464101615137755) -- (7,-3.464101615137755);
\draw[dashed, draw opacity=0.2] (-7,-2.3094010767585034) -- (7,-2.3094010767585034);
\draw[dashed, draw opacity=0.2] (-7,-1.1547005383792517) -- (7,-1.1547005383792517);
\draw[dashed, draw opacity=0.2] (-7,0.0) -- (7,0.0);
\draw[dashed, draw opacity=0.2] (-7,1.1547005383792517) -- (7,1.1547005383792517);
\draw[dashed, draw opacity=0.2] (-7,2.3094010767585034) -- (7,2.3094010767585034);
\draw[dashed, draw opacity=0.2] (-7,3.464101615137755) -- (7,3.464101615137755);
\draw[dashed, draw opacity=0.2] (-7,4.618802153517007) -- (7,4.618802153517007);
\draw[dashed, draw opacity=0.2] (-7,5.773502691896258) -- (7,5.773502691896258);
\draw[dashed, draw opacity=0.2] (-7,6.92820323027551) -- (7,6.92820323027551);
\draw[dashed, draw opacity=0.2] (6.625214782339287,-7) -- (7,-6.3508529610858835);
\draw[dashed, draw opacity=0.2] (5.291881449005953,-7) -- (7,-4.041451884327381);
\draw[dashed, draw opacity=0.2] (3.95854811567262,-7) -- (7,-1.7320508075688774);
\draw[dashed, draw opacity=0.2] (7,0.5773502691896258) -- (2.6252147823392864,-7);
\draw[dashed, draw opacity=0.2] (1.2918814490059531,-7) -- (7,2.886751345948129);
\draw[dashed, draw opacity=0.2] (7,5.196152422706632) -- (-0.04145188432738026,-7);
\draw[dashed, draw opacity=0.2] (6.708118550994047,7) -- (-1.3747852176607136,-7);
\draw[dashed, draw opacity=0.2] (5.374785217660713,7) -- (-2.708118550994047,-7);
\draw[dashed, draw opacity=0.2] (-4.041451884327381,-7) -- (4.041451884327381,7);
\draw[dashed, draw opacity=0.2] (-5.374785217660713,-7) -- (2.708118550994047,7);
\draw[dashed, draw opacity=0.2] (1.3747852176607136,7) -- (-6.708118550994047,-7);
\draw[dashed, draw opacity=0.2] (0.04145188432738026,7) -- (-7,-5.196152422706632);
\draw[dashed, draw opacity=0.2] (-7,-2.886751345948129) -- (-1.2918814490059531,7);
\draw[dashed, draw opacity=0.2] (-2.6252147823392864,7) -- (-7,-0.5773502691896258);
\draw[dashed, draw opacity=0.2] (-3.95854811567262,7) -- (-7,1.7320508075688774);
\draw[dashed, draw opacity=0.2] (-5.291881449005953,7) -- (-7,4.041451884327381);
\draw[dashed, draw opacity=0.2] (-6.625214782339287,7) -- (-7,6.3508529610858835);
\end{tikzpicture}
\end{center}
\vspace{\baselineskip}
\caption{The spherical unitary dual of split $G_2(\mb{R})$ ({\color{unitary} green}), with the bounds given by the Dirac inequality ({\color{Cyan} blue}) and the FPP Conjecture ({\color{Red} red}).} \label{fig:G2}
\end{wrapfigure}
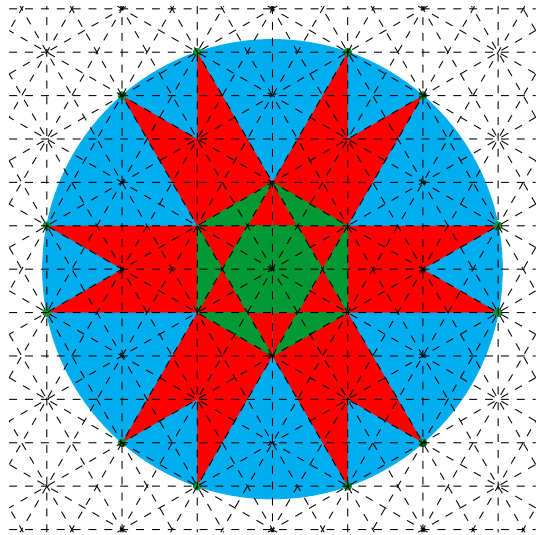

In his pioneering work on spherical representations, Barbasch, partly joint with Ciubotaru, proposed (and proved in some cases) a \emph{polyhedral} bound on the unitary dual (\cite{BCSpherical}). For real reductive groups, these predictions were sharpened rather significantly by Adams, van Leeuwen, Miller, and Vogan (\cite{VoganFPPtalk}). Their refined bound is called the `FPP conjecture' since it is formulated in terms of the \emph{fundamental parallelepiped}, i.e. the region in $\fh^*_{\RR}$ defined by the inequalities $0 \leq \langle \lambda,\check\alpha \rangle \leq 1$ for simple co-roots $\check\alpha$. Roughly, the FPP conjecture states that if $\pi$ is an irreducible unitary representation, then one of the following is true:
\begin{itemize}
    \item[(1)] the \emph{infinitesimal character} of $\pi$ belongs to the fundamental parallelepiped, or
    \item[(2)] $\pi$ is \emph{cohomologically induced} (in the weakly good range) from a unitary representation which satisfies the property in (1).
\end{itemize}
A more precise statement will be given in the following subsection. For most $(P,\sigma)$, the FPP conjecture gives a dramatically sharper bound than the Dirac inequality. Figure \ref{fig:G2} illustrates the difference between these two bounds in the case of the spherical unitary dual $U(P_0,1)$ of split $G_2(\RR)$. Perhaps more importantly, the FPP conjecture suggests a natural way of \emph{organizing} the unitary dual. Namely, it implies that
$$\Pi_u(G(\RR)) = \bigcup_{Q=LU} \mathrm{CohInd}^{\fg}_{\fq}\left(\Pi_{u,\mathrm{FPP}}(L(\RR))\right),$$
where $\Pi_{u,\mathrm{FPP}}(L(\RR))$ is the set of irreducible unitary $L(\RR)$-representations satsifying the FPP condition, $\mathrm{CohInd}^{\fg}_{\fq}$ is cohomological induction, and the union runs over the finite set of $K$-conjugacy classes of $\theta$-stable parabolics of $G$. We note that Adams, van Leeuwen, Miller, and Vogan are in the process of developing an efficient algorithm in {\tt atlas} for computing the set $\Pi_{u,\mathrm{FPP}}(L(\RR))$ for any real reductive group $L(\RR)$. Thus, the FPP conjecture may be the key missing ingredient in the computation of the unitary dual of a real reductive group.

The FPP conjecture is known in several special cases, including for spherical representations of split groups (\cite{BCSpherical}), principal series representations of split $E_8(\RR)$ (announced in \cite{MillerE8talk}), as well as all representations of complex groups (\cite{WongDongFPP}) and indefinite unitary groups (\cite{Wong2024}). In this paper, we will prove the FPP conjecture in general. Our argument is completely case-free and also applies to non-linear coverings of real reductive groups.

\subsection{FPP Conjecture}

From now on let $G$ be a complex connected reductive algebraic group with Lie algebra $\mf{g}$. Write $H$ for the abstract Cartan and $\mf{h} = \mrm{Lie}(H)$. If we choose a maximal torus and Borel subgroup $T \subset B \subset G$, then there are natural identifications $H \cong B/N \cong T$, where $N$ is the unipotent radical of $B$. We thus obtain roots and coroots
\[ \Phi \subset \mb{X}^*(H) := \Hom(H, \mb{C}^\times) \quad \text{and} \quad \check\Phi \subset \mb{X}_*(H) := \Hom(\mb{C}^\times, H)\]
and a Weyl group $W = N_G(T)/T$ acting on $T = H$. We write $\Phi_+, \Phi_- \subset \Phi$ for the corresponding sets of positive and negative roots; here we adopt the convention that the negative roots $\Phi_-$ are the weights of $T$ acting on $\mrm{Lie}(N)$ and $\Phi_+ = -\Phi_-$. The based root datum constructed in this way is canonically independent of the choice of $T$ and $B$. 

Recall the \emph{Harish-Chandra isomorphism}
\begin{align*}
\mf{h}^*/W &\cong \Hom(\mathfrak{Z}(\mf{g})), \mb{C}) \\
\lambda &\mapsto \chi_\lambda
\end{align*}
where $\fZ(\fg)$ is the center of the universal enveloping algebra $U(\fg)$ of $\fg$. This bijection is normalized so that $\rho = \frac{1}{2}\Phi^+ \in \fh^*/W$ corresponds to the infinitesimal character of the trivial representation. We say that $\chi_{\lambda}$ is \emph{real} if $\lambda$ belongs to the canonical real form $\fh_{\RR}^* := \mb{X}^*(H) \otimes_{\ZZ} \RR$ of $\fh^*$.

Now let $G(\RR)$ be a real form of $G$ and let $G_{\RR}$ be a Lie group equipped with a finite map $G_{\RR} \to G$ onto an open subgroup of $G(\RR)$. Choose a Cartan involution $\theta$ of $G$ compatible with $G(\RR)$, and let $K_{\RR}$ be the preimage of $G^{\theta}$ under the map $G_{\RR} \to G$. Then $K_{\RR}$ is a maximal compact subgroup of $G_{\RR}$ and its complexification $K$ admits a finite surjective map onto a open subgroup subgroup of $G^{\theta}$. We write $\fg_{\RR}$ (resp. $\fk_{\RR}$) for the real Lie algebra of $G_{\RR}$ (resp. $K_{\RR}$) and $\fg$ (resp. $\fk$) for the complex Lie algebras of $G$ (resp. $K$). 

Recall that an irreducible $(\fg,K)$-module is \emph{unitary} if it admits a positive-definite $(\fg_{\RR},K_{\RR})$-invariant Hermitian form. The relation to the discussion in the previous subsection is the following: by a result of Harish-Chandra, there is a natural bijection between $\Pi_u(G_{\RR})$ and the set of isomorphism classes of irreducible unitary $(\fg,K)$-modules.

To state the FPP conjecture, we will need to recall the notion of \emph{cohomological induction} (see, for example, \cite[Chapter V]{KnappVogan1995}). Let $P \subset G$ be a $\theta$-stable parabolic subgroup with $\theta$-stable Levi factor $L$. Write $K_P$ and $K_L$ for the preimages of $P$ and $L$ in $K$. The forgetful functor
$$M(\fg,K) \to M(\fg, K_L)$$
has a left adjoint $\Pi$ with left derived functors $\Pi_j$. If $M_L \in M(\fl, K_L)$, we define the \emph{$j$th cohomological induction}
$$\mathcal{L}_jM_L := \Pi_j(U(\fg) \otimes_{U(\fp)} (M_L \otimes \mathrm{det}(\fg/\fp))) \in M(\fg,K),$$
where $M_L$ is regarded as a $U(\fp)$-module via the natural quotient map $\fp \to \fl$. The functors $\mathcal{L}_j$ take finite-length $(\fl,K_L)$-modules to finite-length $(\fg,K)$-modules. Moreover, if $M_L$ has infinitesimal character $\chi_{\lambda}$, then $\mathcal{L}_jM_L$ has infinitesimal character $\chi_{\lambda+\rho_P}$, where $\rho_P=\rho - \rho_L \in \fh^*$ is the half-sum of the roots in $\fg/\fp$.

Now suppose that $M_L$ is an irreducible $(\fl,K_L)$-module. We say that $M_L$ is in the \emph{weakly good range} if
$$\langle \lambda + \rho_P, \alpha^{\vee} \rangle \geq 0, \qquad \forall \alpha \in \Phi(\fh,\fg/\fp).$$
In this case, we have that $\mathcal{L}_jM_L = 0$ for $j \neq S := \dim(K/K_P)$ and that $\mathcal{L}_SM_L$ is irreducible or zero (\cite[Theorem 1.2]{Vogan1984}). Moreover, $\mathcal{L}_SM_L$ is unitary if and only if $M_L$ is (\cite[Theorem 1.3]{Vogan1984}). We say in this case that $\mathcal{L}_SM_L$ is \emph{cohomologically induced from $L$ in the weakly good range}.

\begin{theorem}[FPP Conjecture, see Theorem \ref{thm:body FPP} below] \label{thm:intro FPP}
Let $M$ be an irreducible unitary $(\fg,K)$-module of real infinitesimal character $\chi_{\lambda}$. Conjugating by $W$, we can assume that $\lambda$ is dominant for $G$. Suppose that there is a simple co-root $\alpha^{\vee}$ such that
$$\langle \lambda, \alpha^{\vee} \rangle > 1.$$
Then $M$ is cohomologically induced in the weakly good range from a proper Levi subgroup.
\end{theorem}

We note that if $G_{\RR}$ is a simple noncompact group and $M$ is the Langlands quotient of a minimal principal series representation, then $M$ cannot be cohomologically induced in the weakly good range (this follows immediately from Proposition \ref{prop:coh induction orbit} below). So in this case Theorem \ref{thm:body FPP} states that the infinitesimal character of $M$ is contained in the FPP.

Our proof of Theorem \ref{thm:intro FPP} is an application of the Hodge-theoretic approach to unitary representations of real groups, conjectured by Schmid and Vilonen \cite{SchmidVilonen2011} and developed into a working theory in \cite{DV}. The cornerstone of this approach is the unitarity criterion \cite[Theorem 5.3]{DV} (see Theorem \ref{thm:unitarity criterion} below for a precise statement), which gives an infinite sequence of conditions that must be satisfied by any unitary representation. Our proof shows that the FPP bound arises from the first of these conditions alone.

\subsection{Acknowledgments}

We would like to thank David Vogan for explaining the FPP conjecture to us in detail. We are also grateful to Jeff Adams, Stephen Miller, and Kari Vilonen for numerous helpful discussions.

\section{Hodge theory and unitarity}

In this section, we briefly recall the salient points of Beilinson-Bernstein localization and the Hodge-theoretic approach to unitarity.

\subsection{Beilinson-Bernstein localization}

Let $\mc{B}$ denote the flag variety of the complex reductive group $G$, i.e., the complex variety of Borel subgroups of $G$. We have a canonical identification
\begin{equation} \label{eq:pic}
\begin{aligned}
\mb{X}^*(H) &\cong \mrm{Pic}^G(\mc{B}) \\
\mu &\mapsto \mc{O}(\mu)
\end{aligned}
\end{equation}
between the Picard group of $G$-equivariant line bundles on $\mc{B}$ and the character lattice of the abstract Cartan $H$. We fix our conventions so that the dominant weights in $\mb{X}^*(H)$ correspond to the semi-ample line bundles on $\mc{B}$.

The basic construction of Beilinson-Bernstein theory \cite{BB1} attaches to every $\lambda \in \mf{h}^*$ a sheaf of rings $\mc{D}_\lambda$ on $\mc{B}$, called a \emph{sheaf of twisted differential operators}, which is locally (but not globally) isomorphic to the ordinary sheaf of algebraic differential operators on $\mc{B}$. We arrange our conventions so that if $\mu \in \mb{X}^*(H)$ then $\mc{D}_{\mu + \rho} = \mrm{Diff}(\mc{O}(\mu))$ is the sheaf of differential operators on the line bundle $\mc{O}(\mu)$, where $\rho \in \mf{h}^*$ is half the sum of the positive roots of $G$. In general, we have an equivalence of categories
\[ \mc{O}(\mu) \otimes - \colon \Mod(\mc{D}_\lambda) \xrightarrow{\sim} \Mod(\mc{D}_{\lambda + \mu}) \]
for $\lambda \in \mf{h}^*$ and $\mu \in \mb{X}^*(H)$.

The key property of $\mc{D}_\lambda$ is that its global sections
\[ \Gamma(\mc{B}, \mc{D}_\lambda) = \frac{U(\mf{g})}{\langle x - \chi_\lambda(x) \mid x \in \mf{Z}(\mf{g})\rangle}\]
are identified with the central quotient of $U(\mf{g})$ in which $\mf{Z}(\mf{g})$ acts through the character $\chi_\lambda$. In particular, we have a global sections functor
\[ \Gamma \colon \Mod(\mc{D}_\lambda) \to \Mod(U(\mf{g}))_{\chi_\lambda},\]
where $\Mod(\mc{D}_\lambda)$ is the category of quasi-coherent sheaves of $\mc{D}_\lambda$-modules on $\mc{B}$ and $\Mod(U(\mf{g}))_{\chi_\lambda}$ is the category of $\mf{g}$-modules with infinitesimal character $\chi_\lambda$. Similarly, for any algebraic group $K$ equipped with a homomorphism $K \to G$, we obtain a functor
\[ \Gamma \colon \Mod(\mc{D}_\lambda, K) \to \Mod(\mf{g}, K)_{\chi_\lambda},\]
where $\Mod(\mc{D}_\lambda, K)$ is the category of (strongly) $K$-equivariant objects in $\Mod(\mc{D}_\lambda)$ and $\Mod(\mf{g}, K)_{\chi_\lambda}$ is the category of $(\mf{g}, K)$-modules with infinitesimal character $\chi_\lambda$.

The following fundamental fact is a standard corollary of the main theorem of \cite{BB1}. For the statement, we recall that $\lambda \in \mf{h}^*$ is called \emph{integrally dominant} if $\langle \lambda, \check\alpha \rangle \not\in \mb{Z}_{<0}$ for all positive coroots $\check\alpha \in \check\Phi_+$.

\begin{thm}[Beilinson-Bernstein]
Assume that $\lambda$ is integrally dominant and let $\mc{M} \in \Mod(\mc{D}_\lambda, K)$ be irreducible. Then either $\Gamma(\mc{M}) = 0$ or $\Gamma(\mc{M}) \in \Mod(\mf{g}, K)_{\chi_\lambda}$ is irreducible. Moreover, this defines a bijection
\[ \Gamma \colon \left\{\left.\begin{matrix} \mc{M} \in \Mod(\mc{D}_\lambda, K) \\ \text{irreducible} \end{matrix}\,\right|\, \Gamma(\mc{M}) \neq 0 \right\} \xrightarrow{\sim} \{ M \in \Mod(\mf{g}, K)_{\chi_\lambda} \; \text{irreducible}\},\]
where both sides are taken up to isomorphism.
\end{thm}

Let us comment briefly on the classification of the irreducible objects in $\Mod(\mc{D}_\lambda, K)$ in our setting of interest, where $K$ is the complexification of a maximal compact subgroup $K_\mb{R} \subset G_\mb{R}$. In this case, the group $K$ acts on the flag variety $\mc{B}$ with finitely many orbits, so standard $\mc{D}$-module theory implies that every irreducible $\mc{M} \in \Mod(\mc{D}_\lambda, K)$ is of the form
\[ \mc{M} = j_{!*}\gamma \]
for $j \colon \mb{O} \to \mc{B}$ the inclusion of a $K$-orbit and $\gamma$ an irreducible twisted local system on $\mb{O}$. Here $j_{!*}$ denotes the intermediate extension for $\mc{D}$-modules. We will refer to $\mb{O}$ as the \emph{$K$-orbit supporting $\mc{M}$}.

\subsection{Cohomological induction and localization}

The Beilinson-Bernstein description of irreducible $(\mf{g}, K)$-modules interacts in a very simple way with cohomological induction. We refer the reader to \cite[\S 8.2]{DV} for a brief summary. We will need only the following simple criterion for an irreducible $(\mf{g}, K)$-module to be cohomologically induced in the weakly good range.

\begin{prop} \label{prop:coh induction orbit}
Assume that $\lambda \in \mf{h}^*_\mb{R}$ is dominant and let $\mc{M} \in \Mod(\mc{D}_\lambda, K)$ be an irreducible object with supporting $K$-orbit $\mb{O} \subset \mc{B}$ such that $M = \Gamma(\mc{M}) \neq 0$. Let $\mc{P}$ be a partial flag variety parametrizing a $\theta$-stable conjugacy class of parabolic subgroups and let $\pi \colon \mc{B} \to \mc{P}$ be the natural projection. Then $M$ is cohomologically induced from a $\theta$-stable parabolic $P$ in the conjugacy class $\mc{P}$ if and only if $\pi(\mb{O}) \subset \mc{P}$ contains the corresponding $\theta$-fixed point $x$.
\end{prop}
\begin{proof}
This is an immediate corollary of \cite[Proposition 8.4]{DV}. We give the argument for the ``if'' direction as this is the only one we will use. Suppose that $\pi(\mb{O})$ contains a $\theta$-fixed point $x$ corresponding to a $\theta$-stable parabolic $P$ with Levi factor $L$. Then we must have $\pi(\mb{O}) = K \cdot x$, a closed $K$-orbit, and hence that $\mc{M}$ is supported on $\pi^{-1}(K\cdot x)$. So we may write
\[ \mc{M} = i_*\mc{M}' \quad \text{for some $\mc{M}' \in \Mod(\mc{D}_{\pi^{-1}(K \cdot x), \lambda}, K)$},\]
where $i \colon \pi^{-1}(K \cdot x) \to \mc{B}$ is the inclusion. Restricting to the fiber $\mc{B}_L = \pi^{-1}(x)$, we obtain an irreducible $\mc{D}_{\mc{B}_L, \lambda - \rho_P}$-module $\mc{M}_L$; here the shift by $\rho_P$ comes from the different $\rho$-shifts in the labelling for twisted differential operators on $\mc{B}_L$ versus $\mc{B}$. Setting $M_L = \Gamma(\mc{B}_L, \mc{M}_L)$, we have by \cite[Proposition 8.4]{DV} that
\[ M = \mc{L}_S(M_L)\]
is cohomologically induced from $M_L$. Since $\lambda = \lambda - \rho_P + \rho_P$ is dominant by assumption, the cohomological induction is in the weakly good range.
\end{proof}

\subsection{Mixed Hodge modules and localization}

Finally, let us recall some of the Hodge-theoretic picture of \cite{SchmidVilonen2011}, as developed in \cite{DV}.

First, if $\lambda \in \mf{h}^*_\mb{R} := \mb{X}^*(H) \otimes \mb{R}$, then there are categories $\mhm(\mc{D}_\lambda)$ and $\mhm(\mc{D}_\lambda, K)$ of (equivariant) mixed Hodge $\mc{D}_\lambda$-modules, equipped with forgetful functors 
\[ \mhm(\mc{D}_\lambda) \to \Mod(\mc{D}_\lambda, F_\bullet)  \quad \text{and} \quad \mhm(\mc{D}_\lambda, K) \to \Mod(\mc{D}_{\lambda}, F_\bullet, K) \]
to the categories of (equivariant) $\mc{D}_\lambda$-modules equipped with a good filtration compatible with the order filtration $F_\bullet \mc{D}_\lambda$. For $\mc{M} \in \mhm(\mc{D}_\lambda)$, the associated filtration $F_\bullet \mc{M}$ is called the \emph{Hodge filtration}; it is an increasing filtration by $\mc{O}_\mc{B}$-coherent subsheaves, which is locally finitely generated over $F_\bullet \mc{D}_\lambda$. We note that, essentially by definition (see \cite[\S 2.2]{DV}), we have
\[ F_p(\mc{M} \otimes \mc{O}(\mu)) = (F_p\mc{M}) \otimes \mc{O}(\mu) \quad \text{for $\mu \in \mb{X}^*(H)$}. \]

One of the main theorems of \cite{DV} is:

\begin{thm}[{\cite[Theorem 5.3]{DV}}] \label{thm:global generation}
Assume that $\lambda \in \mf{h}^*_\mb{R}$ is dominant (i.e., $\langle \lambda, \check\alpha \rangle \geq 0$ for all positive coroots $\check\alpha$). If $\mc{M} \in \mhm(\mc{D}_\lambda)$ is generated by global sections as a $\mc{D}_\lambda$-module, then so is its Hodge filtration $F_\bullet\mc{M}$, i.e.,
\[ F_p\mc{M} = \sum_{p_1 + p_2 = p} F_{p_1}\mc{D}_\lambda \cdot \Gamma(F_{p_2}\mc{M}).\]
\end{thm}

Since the sheaves $F_p\mc{D}_\lambda$, being quotients of $F_p U(\mf{g}) \otimes \mc{O}_\mc{B}$, are themselves globally generated as $\mc{O}_{\mc{B}}$-modules, Theorem \ref{thm:global generation} implies that each piece $F_p\mc{M}$ is globally generated as an $\mc{O}_\mc{B}$-module.

It is a basic fact, which follows from the classification in terms of local systems on $K$-orbits, that every irreducible $\mc{M} \in \Mod(\mc{D}_\lambda, K)$ lifts to uniquely (up to tensoring with a $1$-dimensional Hodge structure) to an object in $\mhm(\mc{D}_\lambda, K)$. In particular, every such $\mc{M}$ is equipped with a Hodge filtration $F_\bullet \mc{M}$, well-defined up to a shift in indexing. This induces a Hodge filtration $F_\bullet M := \Gamma(F_\bullet \mc{M})$ on the $(\mf{g}, K)$-module $M = \Gamma(\mc{M})$. Since every irreducible $(\mf{g}, K)$-module $M$ with real infinitesimal character is uniquely of this form for some $\lambda \in \mf{h}^*_\mb{R}$ dominant and $\mc{M} \in \mhm(\mc{D}_\lambda, K)$ irreducible, this gives a canonical (up to shift) Hodge filtration on every such $M$.

The key feature of the Hodge filtration $F_\bullet M$ is that it can be used to determine the unitarity of $M$. Recall the following fact, originally due to Adams, van Leeuwen, Trapa and Vogan \cite{ALTV}; for the statement, we say that an irreducible $(\mf{g}, K)$-module is \emph{Hermitian} if it admits a non-degenerate $(\mf{g}_\mb{R}, K_\mb{R})$-invariant Hermitian form.

\begin{lem}
Let $M$ be an irreducible $(\mf{g}, K)$-module with real infinitesimal character. Then $M$ is Hermitian if and only if there exists a $K$-equivariant involution $\theta \colon M \to M$ such that
\[ \theta(x\cdot m) = \theta(x) \cdot \theta(m) \quad \text{for $x \in \mf{g}$ and $m \in M$},\]
where $\theta \colon \mf{g} \to \mf{g}$ is the derivative of the Cartan involution on $G$.
\end{lem}

The Cartan involution on an irreducible Hermitian $(\mf{g}, K)$-module is of course unique up to sign. Geometrically, it comes from an isomorphism
\[ \theta^*\mc{M} \cong \mc{M},\]
where $\theta^*\mc{M}$ is the pullback of $\mc{M}$ along $\theta \colon \mc{B} \to \mc{B}$. Note that for $\mc{M} \in \Mod(\mc{D}_\lambda, K)$, we have $\theta^*\mc{M} \in \Mod(\mc{D}_{\delta\lambda}, K)$, where $\delta \colon H \to H$ is the automorphism of the based root datum induced by $\theta^* \colon \mrm{Pic}^G(\mc{B}) \to \mrm{Pic}^G(\mc{B})$ and the isomorphism \eqref{eq:pic}. In particular, we must have $\delta \lambda = \lambda$ if $M$ is Hermitian. Since the Hodge filtration is uniquely determined by the underlying $\mc{D}_\lambda$-module and the index of its first non-zero piece, it follows that $\theta$ preserves the Hodge filtration on $\mc{M}$ and hence $M$.

The main theorem of \cite{DV} concerning unitarity is:

\begin{thm}[{\cite[Theorem 1.8]{DV}}] \label{thm:unitarity criterion}
Assume that $M$ is an irreducible Hermitian $(\mf{g}, K)$-module with real infinitesimal character and Cartan involution $\theta$. Then $M$ is unitary if and only if $\pm \theta$ acts on $\Gr^{F}_p M$ by $(-1)^p$ for all $p$.
\end{thm}

\section{Proof of the FPP Conjecture}

In this section, we prove the following theorem, which implies the FPP conjecture and clarifies its relation to Hodge theory.

\begin{thm} \label{thm:body FPP}
Fix $\lambda \in \mf{h}^*_\mb{R}$ dominant and let $M$ be an irreducible Hermitian $(\mf{g}, K)$-module with real infinitesimal character $\chi_\lambda$. Let $\theta \colon M \to M$ be a Cartan involution and $F_\bullet M$ the Hodge filtration. Assume that 
\begin{enumerate}
\item there exists a simple coroot $\check\alpha$ such that $\langle \lambda, \check\alpha \rangle > 1$ and
\item $M$ is not cohomologically induced in the weakly good range from a proper Levi subgroup.
\end{enumerate}
Then $\theta$ does not act by a scalar on the lowest non-zero graded piece of $\Gr^{F}M$ and hence $M$ is not unitary.
\end{thm}

We first make the following simple reduction. Let $G^{sc}$ be the universal cover of the derived group $[G, G]$, and set $G_\mb{R}^{sc} = G^{sc} \times_G G_\mb{R}$. Then any irreducible (unitary) representation $M$ of $G_\mb{R}$ restricts to a semi-simple (unitary) representation of $G_\mb{R}^{sc}$, whose constituents are cohomologically induced if and only if $M$ is. So we may assume for the sake of convenience that $G$ is simply connected; we will do so from now on. This ensures that the fundamental weights $\varpi_\alpha \in \mf{h}^*_\mb{R}$ all lie in the character lattice $\mb{X}^*(H)$.

\begin{lem} \label{lem:generation}
Let $\lambda \in \mf{h}^*_\mb{R}$ be dominant, let $\mc{M}$ be a $\mc{D}_\lambda$-module on $\mc{B}$ and let $\alpha$ be a simple root with corresponding fundamental weight $\varpi_\alpha$. If $\mc{M}$ is globally generated and $\langle \lambda, \check\alpha\rangle > 1$ then the $\mc{D}_{\lambda - \varpi_\alpha}$-module $\mc{M} \otimes \mc{O}(-\varpi_\alpha)$ is also globally generated. 
\end{lem}
\begin{proof}
The proof is a very slight modification of Beilinson and Bernstein's proof of global generation in the regular case.

Let $L(\varpi_\alpha) = \Gamma(\mc{B}, \mc{O}(\varpi_\alpha))$ be the irreducible algebraic representation of $G$ with highest weight $\varpi_\alpha$. Consider the tautological surjection
\begin{equation} \label{eq:generation 1}
\mc{M} \otimes \mc{O}(-\varpi_\alpha) \otimes L(\varpi_\alpha) \to \mc{M} \otimes \mc{O}(-\varpi_\alpha) \otimes \mc{O}(\varpi_\alpha) = \mc{M}.
\end{equation}
We claim that \eqref{eq:generation 1} splits. To see this, let $\mc{K} = \ker(\mc{O}_\mc{B} \otimes L(\varpi_\alpha) \to \mc{O}(\varpi_\alpha))$ and recall that $\mc{K}$ has a $G$-equivariant filtration by line bundles of the form $\mc{O}(\mu)$ where $\mu \neq \varpi_\alpha$ is a weight of $L(\varpi_\alpha)$. Thus, the kernel $\mc{M} \otimes \mc{O}(-\varpi_\alpha) \otimes \mc{K}$ of \eqref{eq:generation 1} has a filtration, as a sheaf of $U(\mf{g})$-modules, by subquotients of the form
\[ \mc{M} \otimes \mc{O}(\mu - \varpi_\alpha) \quad \text{for $\mu\neq \varpi_\alpha$ a weight of $L(\varpi_\alpha)$.}\]
Now, the center $\mf{Z}(\mf{g}) \subset U(\mf{g})$ acts on $\mc{M} \otimes \mc{O}(\mu - \varpi_\alpha)$ via the character $\chi_{\lambda + \mu - \varpi_\alpha}$ and on $\mc{M}$ via $\chi_{\lambda}$, so to show that \eqref{eq:generation 1} splits, it suffices to show that
\[ \chi_{\lambda + \mu - \varpi_\alpha} \neq \chi_\lambda  \quad \text{(i.e., $\lambda + \mu -\varpi_\alpha \not\in W\lambda$)}\]
for all such $\mu$. Now, we may write
\[ \mu = \varpi_\alpha - \alpha - \sum_i n_i \alpha_i \]
where $n_i \geq 0$ and the $\alpha_i$ are simple roots. By assumption on $\lambda$, $\lambda - \varpi_\alpha$ is dominant and satisfies $(\lambda - \varpi_\alpha,\alpha) > 0$ (where $(\,,\,)$ is the Killing form), so
\[ (\lambda - \varpi_\alpha, \mu) < (\lambda - \varpi_\alpha, \varpi_\alpha).\]
Moreover, since $\mu$ is a weight of $L(\varpi_\alpha)$, we have $\|\mu\|^2 \leq \|\varpi_\alpha\|^2$ and hence
\[ \|\lambda - \varpi_\alpha + \mu\|^2 = \| \lambda - \varpi_\alpha\|^2 + 2(\lambda - \varpi_\alpha, \mu) + \|\mu\|^2 < \|\lambda -\varpi_\alpha\|^2 + 2(\lambda - \varpi_\alpha, \varpi_\alpha) + \|\varpi_\alpha\|^2 = \|\lambda\|^2.\]
So we cannot have $\lambda + \mu - \varpi_\alpha \in W \lambda$, so \eqref{eq:generation 1} splits as claimed.

Since \eqref{eq:generation 1} splits, we deduce that induced map on global sections
\[ \Gamma(\mc{M} \otimes \mc{O}(-\varpi_\alpha)) \otimes L(\varpi_\alpha) \to \Gamma(\mc{M}) \]
is also surjective. Since $\mc{M}$ is globally generated, this implies surjectivity of the induced morphism
\[ \Gamma(\mc{M} \otimes \mc{O}(-\varpi_\alpha)) \otimes L(\varpi_\alpha) \otimes \mc{O}_\mc{B} \to \mc{M} \]
and hence of the morphism
\[ \Gamma(\mc{M} \otimes \mc{O}(-\varpi_\alpha)) \otimes \mc{O}(\varpi_\alpha) \to \mc{M}\]
through which it factors. Tensoring both sides with $\mc{O}(-\varpi_\alpha)$, we see that
\[ \Gamma(\mc{M} \otimes \mc{O}(-\varpi_\alpha)) \otimes \mc{O}_\mc{B} \to \mc{M} \otimes \mc{O}(-\varpi_\alpha)\]
is surjective, which proves the lemma.
\end{proof}

\begin{proof}[Proof of Theorem \ref{thm:body FPP}]
Let us write $M = \Gamma(\mc{M})$ where $\mc{M} \in \mhm(\mc{D}_\lambda, K)$ is irreducible and $\lambda \in \mf{h}^*_\mb{R}$ is dominant. We write $\mb{O} \subset \mc{B}$ for the $K$-orbit supporting $\mc{M}$.

Choose a simple root $\alpha$ as in the statement of the theorem. Recall the involution $\delta \colon H \to H$ induced by $\theta \colon G \to G$; this permutes the set of simple roots in $\mb{X}^*(H)$. We distinguish two cases: either $\delta(\alpha) = \alpha$ or $\delta(\alpha) \neq \alpha$. If $\delta(\alpha) = \alpha$, we let $\pi \colon \mc{B} \to \mc{P}$ be the projection to the partial flag variety parametrizing maximal parabolic subgroups of $G$ for which $\alpha$ is not a root, and we set $\mu = \varpi_\alpha$. If $\delta(\alpha) \neq \alpha$, we let $\pi \colon \mc{B} \to \mc{P}$ be the projection to the partial flag variety parametrizing next-to-maximal parabolic subgroups for which $\alpha, \delta(\alpha)$ are not roots and set $\mu = \varpi_\alpha + \varpi_{\delta(\alpha)}$. Note that in either case $\mc{O}(\mu)$ descends to a very ample line bundle (which we also denote by $\mc{O}(\mu)$) on $\mc{P}$, and that the action of $\theta$ on $\mc{B}$ descends to an action on $\mc{P}$ such that $\theta^*\mc{O}(\mu) \cong \mc{O}(\mu)$. We fix such an isomorphism so that $\theta$ acts as an involution on $\mc{O}(\mu)$.

Choose a point $x$ in $\OO$. Since $M$ is not cohomologically induced in the weakly good range, by Proposition \ref{prop:coh induction orbit}, the point $\pi(x) \in \mc{P}$ is not $\theta$-fixed, i.e., we have $\pi(x) \neq \pi(\theta(x))$. Since $\mc{O}(\mu)$ is very ample on $\mc{P}$, we may therefore choose a section $s \in \Gamma(\mc{P}, \mc{O}(\mu)) = \Gamma(\mc{B},\mc{O}(\mu))$ such that $s(x) \neq 0$ but $s(\theta(x)) = 0$. Setting $s_{\pm} = (1 \pm \theta)(s)$, we see that $s_{\pm}(x) \neq 0$ and $\theta(s_{\pm}) = \pm s_{\pm}$.

Finally, let us suppose for convenience that the Hodge structure on $\mc{M}$ is chosen so that $F_0\mc{M} \neq 0$ and $F_{-1}\mc{M} = 0$. Since $\mc{M}$ is the intermediate extension of a local system on $\mb{O}$, the fiber $(F_0\mc{M})_x$ of the coherent sheaf $F_0\mc{M}$ at $x$ is non-zero (see, e.g., \cite[\S 5.2]{DavisVilonen22}). Now, since $\mc{M}$ is irreducible and $\Gamma(\mc{M}) \neq 0$, $\mc{M}$ must be globally generated as a $\mc{D}_\lambda$-module and hence as an $\mc{O}_\mc{B}$-module. Since $M$ is Hermitian, we have $\delta \lambda = \lambda$, and hence $\langle \lambda, \check\alpha \rangle = \langle \lambda, \delta(\check\alpha)\rangle > 1$. So $\mc{M} \otimes \mc{O}(-\mu)$, and hence $F_0(\mc{M} \otimes \mc{O}(-\mu)) = F_0\mc{M} \otimes \mc{O}(-\mu)$, is also globally generated by Lemma \ref{lem:generation} and Theorem \ref{thm:global generation}. Hence, there exists a section $v \in \Gamma(F_0\mc{M} \otimes \mc{O}(-\mu))$, which we may as well assume to be an eigenvector for $\theta$, such that $v_x \in (F_0\mc{M} \otimes \mc{O}(-\mu))_x \neq 0$. Thus, the sections
\[ v \otimes s_+, v \otimes s_- \in \Gamma(F_0\mc{M}) = F_0 M\]
are both non-vanishing at $x$ and are eigenvectors for $\theta$ with opposite eigenvalues. In particular, $\theta$ does not act on $F_0 M = \Gr^F_0 M$ with a single sign, so $M$ is not unitary by Theorem \ref{thm:unitarity criterion}.
\end{proof}

\begin{sloppypar} \printbibliography[title={References}] \end{sloppypar}

\end{document}